\pdfoutput=1
\documentclass{amsart}
\usepackage{latexsym}
\usepackage{amsfonts}
\newcommand{\C}{\mathbb{C}}
\newcommand{\R}{\mathbb{R}}
\newcommand{\Q}{\mathbb{Q}}

\newcommand{\To}{\rightarrow}

\theoremstyle{plain}
\newtheorem{Thm}{Theorem}
\newtheorem{Cor}{Corollary}

\newtheorem{Prop}[Cor]{Proposition}

\newtheorem{Conj*}{Conjecture}

\theoremstyle{remark}

\newcommand{\J}{{J}acobian\ }
\newcommand{\p}{polynomial\ }
\newcommand{\pk}{{P}inchuk\ }
\newcommand{\nz}{nonzero\ }

\newcommand{\sa}{semi-algebraic\ }

\begin{document}
\title{On the rational real {J}acobian conjecture}
\author{L. Andrew Campbell} 
\address{908 Fire Dance Lane \\
Palm Desert CA 92211 \\ USA}
\email{lacamp@alum.mit.edu}
\begin{abstract}
Jacobian conjectures (that nonsingular implies a global inverse) 
for rational everywhere defined maps of $\R^n$ to itself
are considered, with no requirement for a constant Jacobian 
determinant or a rational inverse. 
The birational case is proved and the Galois case clarified. 
Two known special cases of the Strong Real Jacobian Conjecture 
(SRJC) are generalized to the rational map context. 
 For an invertible map, the associated extension of 
rational function fields must be of odd degree and must have
no nontrivial automorphisms. That disqualifies the Pinchuk counter
 examples to the SRJC as candidates for invertibility. 
\end{abstract}

\maketitle

\section{Introduction and summary of results}\label{intro}

The \J Conjecture (JC) \cite{BCW82,ArnoBook} 
asserts that a polynomial map $F: k^n \To k^n$,
where $k$ is a field of characteristic zero, has a polynomial inverse if  
it is a Keller map \cite{Keller}, which means that its \J determinant, $j(F)$, is a \nz element of $k$.
The JC is still not settled for any $n > 1$ and any specific field $k$ of characteristic zero.

For $k=\R$,  the Strong Real \J Conjecture (SRJC), 
asserts that a polynomial map $F: \R^n \To \R^n$,
has a real analytic  inverse if it is nonsingular, meaning that
$j(F)$, whether constant or not, 
vanishes nowhere on $\R^n$. 
However, Sergey {P}inchuk exhibited a family 
of counterexamples for $n=2$ \cite{Pinchuk}, 
 so the SRJC holds only in 
special cases.

The  Rational Real \J Conjecture (RRJC) is considered here. 
It is the extension of the SRJC to everywhere defined rational 
maps, as well as \p ones. 
Everywhere defined means that each component of the map 
can be expressed as the quotient of two polynomials 
with a nowhere vanishing denominator. 
That rules out rational functions such as $(x^4+y^4)/(x^2+y^2)$, 
which is not defined at the origin, even though it has a unique 
continuous extension to all of $\R^2$.  
That requirement is crucial, as 
$F= (x^2y^6+2xy^2, xy^3+1/y)$ is Keller and 
maps $(1,1)$ and $(-3,-1)$ to the same point 
\cite{Vitushkin}. 
Assume  
$F: \R^n \To \R^n$ is such a map and is nonsingular. 
Then all its fibers are finite  of size at most the degree of the associated 
finite algebraic extension of rational function fields. 
If $F$ also has a (necessarily real analytic) inverse, then the 
function field extension is of odd degree and has a trivial 
automorphism group. 
The extension degree and the maximum fiber size are of the 
same parity. 
If odd maximum fiber size is added as an additional 
hypothesis to the RRJC or SRJC, it disqualifies the 
\pk counterexamples, for all of which that size is $2$ 
\cite{aspc}. 
If the extension degree is $1$ (the birational case), then
$F$ has an inverse that is also  an everywhere defined
birational nonsingular map. If the extension is {G}alois,
then $F$ has an inverse if, and only if, $F$ is birational.

If the
automorphism group condition is added as an
additional  hypothesis, the RRJC and SRJC are  
 true in the {G}alois case.  
Thus if both necessary conditions are assumed, 
the resulting modified RRJC and SRJC 
conjectures are true in the 
birational and {G}alois cases and 
have no obvious counterexamples.

Finally, two  known special cases of the SRJC 
 are generalized to the RRJC context. 
They show that $F$ is invertible if $A(F)$, the set of points in 
the codomain over which $F$ is not locally a trivial fibration, 
 either  is of codimension  greater than $2$, or  does 
not intersect the image of $F$.

\section{Basic properties}\label{basic}

Both the \J hypothesis and the conclusion of the RRJC can be
restated in various equivalent ways. Principally, the former is equivalent to
the assertion that $F$ is locally diffeomorphic or locally real bianalytic, and the latter to the assertion that $F$ is injective 
or bijective or a homeomorphism or a diffeomorphism. 
These are all obvious, except for the key result that injectivity, also called univalence, implies bijectivity
for maps of $\R^n$ to itself that are
polynomial or, more generally, rational and defined on
all of $\R^n$ \cite{InjectiveReal}.
That result does not generalize to semi-algebraic maps of $\R^n$ to itself \cite{SurjectiveSemialgebraic}. 
Clearly any global univalence theorems \cite{GlobU}
for local diffeomorphisms can yield special cases of the conjecture. 
Properness suffices, and related topological considerations 
play a role below. 
But the focus of this article is on results or conjectures that require
the polynomial or rational character of a map
and involve properties of the associated 
extension of rational function fields.

The extension of function fields exists, and is algebraic of 
finite degree, for any dominant rational $F: \R^n \To \R^n$, 
whether defined everywhere or not. $F$ is dominant if, 
and only if, $j(F)$ is not identically zero. 
The extension is the inclusion of the subfield generated 
over $\R$ by 
the (algebraically independent) components of $F$ in the
rational function field on the domain of $F$,
and will be written as $\R(F) \subseteq \R(X)$ 
or $\R(X)/\R(F)$. 
The degree $d$ of the extension is called the extension degree 
of $F$. If $F$ is generically $N$-to-one for a positive 
integer $N$, then $N$ is called the geometric degree of $F$. 
In general, let  $t \in \R(X)$ be a primitive element 
for the extension, meaning that $\R(F)(t)=\R(X)$. 
For generic $y$ in the codomain, inverse images $x$ of $y$ correspond bijectively to real roots $r=t(x)$ at $y$ 
of the monic minimal polynomial of $t$ over $\R(F)$. 
So a generic fiber of $F$ is finite and either empty 
or of positive size at most $d$, 
but $F$  need not have a geometric degree.

By definition, an automorphism of the extension is a 
field automorphism of $\R(X)$ that fixes every element 
of $\R(F)$.

\begin{Prop}\label{needs}
If the geometric degree of $F$ is $1$, then the extension 
has odd degree and trivial automorphism group.
\end{Prop}

\begin{proof}
The nonreal roots occur in complex conjugate pairs, and the 
degree of the monic minimal polynomial for a primitive 
element is $d$. If $G: \R^n \To \R^n$ is the geometric 
realization of an automorphism of $\R(X)$ as a 
rational map and every element of $\R(F)$ is fixed by the 
automorphism, then $F \circ G = F$.  
For a generic $x$, $G$ is defined at $x$, and $F$ 
is defined and locally diffeomorphic at both $x$ and 
$x'=G(x)$.  
Since the geometric degree of $F$ is $1$, $G$ is the 
identity on an open set and therefore, because it is rational, 
the identity map. So the automorphism is also the identity. 
\end{proof}

A map $F: \R^n \To \R^n$ will be called a rational 
nonsingular (nondegenerate) map if it is an everywhere 
defined rational map and $j(F)$ vanishes nowhere 
(resp., does not vanish identically).  
In either case, both of the Proposition \ref{needs}
conclusions become necessary conditions for 
the existence of an inverse.  
The Pinchuk counterexamples \cite{Pinchuk} to the SRJC 
(and hence to the RRJC) are nonsingular \p maps of 
$\R^2$ to $\R^2$ with no inverse. 
All these \pk maps have the same nonconstant, everywhere positive 
\J determinant, 
geometric degree $2$, no point with more 
than $2$ inverse images, exactly $2$ points omitted 
in the image plane, and the 
same extension of degree $6$ 
with trivial automorphism group 
\cite{aspc,PMFF}.

All three conjectures 
discussed are true in the 
 dimension $n=1$ case $f: \R \To \R$. 
 In the JC case, $f$ is of degree $1$. In the SRJC 
case, $f$ is proper, since any nonconstant polynomial becomes infinite 
when its argument does. 
In the RRJC case, $f$ is monotone increasing or decreasing, 
hence injective, thus surjective, so unbounded above and 
below, and therefore proper.

In the RRJC context,
the distinction between \nz constant and nowhere vanishing \J 
determinants is not as critical as it may seem. 
If $F:\R^n \To \R^n$ satisfies the hypotheses,
let $x \in \R^n, z \in \R$ and define 
$F^+:\R^{n+1} \To \R^{n+1}$
by $F^+(x,z)=(F(x),z/(j(F)(x))))$.
Then $F^+$ also  satisfies the hypotheses, $j(F^+)=1$, and 
$F^+$ is injective if, and only if, $F$ is injective. 
As pointed out  in 
\cite{RealJC+SamuelsonMaps}, choosing 
{P}inchuk maps for $F$ yields Keller 
counterexamples to the RRJC in dimension $n=3$.

A Samuelson map is a map with a square \J matrix, 
all of whose leading principal minors, including 
its determinant, vanish nowhere. 
A rational Samuelson map defined on all 
of $\R^n$  has an inverse \cite{RationalSamuelson}, 
which is necessarily Nash  (\sa and real analytic),
but is rational if, and only if, the 
function field extension is birational
(cf. section 3).
The well known real analytic example $(e^x-y^2+3,4y e^x-y^3)$ in 
\cite{Gale-Nikaido} shows that a 
Samuelson map need not be globally
injective (consider $(0,2)$ and $(0,-2)$).
The variation $F(x,y)=(h-y^2+3,4yh-y^3)$ in
\cite{RealJC+SamuelsonMaps}, 
where $h$ is the function 
$h(x)=x+\sqrt{1+x^2}$ (positive 
square root intended) 
has the same properties and is Nash as well. 
So does $F^+$, which is also Keller.

Let $F: \R^n \To \R^n$ be a rational nonsingular map.
It is a local diffeomorphism, hence an open map.
Let $x \in \R^n$ and $y= F(x) \in \R^n$ and define $m(x)$
to be the number of inverse images of $y$ under $F$, 
potentially allowing $+\infty$ as a possible value. 
Since $F$ is open, $m(x')\ge m(x)$ for $x'\in \R^n$ in a neighborhood 
of $x$. So if $A \subseteq \R^n$, the maximum value of $m$ 
on $A$ is also the maximum value of $m$ on its topological 
closure $\bar{A}$. 
So all fibers of $F$, not just generic ones, are finite of 
size at most $d$, where $d$ is the extension degree of $F$. 
The fiber size maximum, $N$, is attained on an open 
subset of the codomain, which must contain a point 
where  $N$ is the number of real roots of a polynomial 
of degree $d$ with real coefficients. 
Thus $N$ and $d$ have the same parity. 
Note that if $N$ and $d$ are odd, then a generic fiber 
is nonempty because a real polynomial of odd degree 
has at least one real root, and so $F(\R^n)$ is a connected 
dense open \sa  subset of the codomain. 
All subsets of $\R^n$ that can be described in the first order logic 
of ordered fields are semi-algebraic. 
The description can include real constant symbols 
(coefficients, values, etc.) and quantification over real 
variables (but not over subsets, functions or natural numbers); 
results for any dimension $n>0$ and involving 
polynomials of arbitrary degrees follow from 
schemas specifying first order descriptions for any fixed 
choice of the natural number parameters. 
As a first application of that principle,  the $N$ subsets of the 
domain $\R^n$ on which $m(x)$ has a 
specified numeric value in the range $1,\ldots,N$,
and the  $N+1$ subsets of the codomain $\R^n$ on which $y$ has a 
specified number of inverse images in the range $0,\ldots,N$, 
are all semi-algebraic. 
By definition, $F$ is proper at a point $y$ in its codomain 
if $y$ has an open neighborhood $U$, such that any compact 
subset of $U$ has a compact inverse image under $F$. 
The set of points $y$ in the codomain at which $F$ is proper is 
readily verified to be the open set of points at which the number of 
inverse images of $y$ is locally constant. 
That set 
contains all points with $N$ inverse images and has an $\epsilon$-ball
first order description.
Its complement $A(F)$, the asymptotic variety of $F$, is therefore 
closed semi-algebraic and the inclusion $A(F)\subset \R^n$ is strict. 
$A(F)$ Is the union for $i=0,\ldots,N-1$ of the  semi-algebraic sets consisting of points $y$ in the codomain at which $F$ 
is not proper and for which $y$ has exactly $i$ inverse images. 
At an interior point $y$ of one of these sets $F$ would be proper, 
contradicting $y \in A(F)$. Thus each such set has empty interior, 
hence is of dimension less than $n$. 
Consequently $\dim A(F) < n$. 
It follows that the complement of $A(F)$ is a finite union of disjoint 
connected open semi-algebraic subsets of $\R^n$ on each of which the number of 
inverse images of points is a constant, with possibly differing 
constants for different connected components.
If   $U$ is  any such  connected component that intersects 
$F(\R^n)$,
then $F^{-1}(U)$ is 
nonempty, open and semi-algebraic. Let  $V$ be one of its 
finitely many connected components. Since $V$ is an open and closed subset of $F^{-1}(U)$, the map $V \To U$ induced 
by $F$ is a proper local homeomorphism of connected, locally compact, and locally arcwise connected spaces and hence it is a covering map. Such a map is surjective, so all of $U$ is 
contained in $F(\R^n)$. 
$V$ must be exactly one of the finitely many connected components 
of the open semi-algebraic set $\R^n \setminus F^{-1}(A(F))$, 
since it is closed in that subset as one element of a finite cover by 
disjoint total spaces of covering maps. 
Speaking informally, this presents a view of $F$ as a finite 
collection of $n$-dimensional covering maps, of possibly 
different degrees, glued together along semi-algebraic sets of 
positive codimension to form $\R^n$ at the total space level, 
whose base spaces, which may sometimes coincide for different 
total spaces, are similarly glued together to form $F(\R^n)$. 
$F(\R^n) \cap A(F)$ is in general neither empty nor all of $A(F)$, 
a behavior exhibited 
by any {P}inchuk map $F$, 
since then $A(F)$ is a \p curve and 
exactly two of its points 
are not in the image of $F$.

\begin{Prop}
If $F: \R^n \To \R^n$ is a rational nonsingular map and 
$F$ is generically injective, then $F$ is invertible and its 
inverse is a nonsingular real analytic map defined on all 
of $\R^n$. 
\end{Prop}

\begin{proof}
Suppose $F$ is injective on a nonempty Zariski open 
set $U \subset \R^n$.  
Let $V$ be the complement of $U$. 
Since $V$ is algebraic and $\dim V < n$, $F(V)$ 
is \sa of maximum dimension at most $n-1$. 
So $F(V)$ is not Zariski dense and therefore the open 
set of points of maximum fiber size $N$ contains a point 
with inverse images only in $U$. 
It follows that $N=1$, that $F$ is injective, and hence 
that $F$ is surjective 
\cite{InjectiveReal}. 
$F$ is locally real bianalytic, and so its global inverse is 
a nonsingular real analytic map. 
\end{proof}

Remark.
The asymptotic variety was defined by Ronen Peretz as the set of finite
limits of a map along curves that tend to infinity
\cite{asympvals,asymptotics}. 
For real \p maps, it can fail to be Zariski closed, and therefore not
technically a variety \cite{GeoPinMap}.
In that context it has been extensively studied by Zbigniew Jelonek
as the set of points at which a map is not proper
\cite{realtrans,geometry}. 
As one result, he shows that
for a nonconstant polynomial map $F: \R^n \To \R^m$, where $n$ and $m$
are any positive integers and no other conditions are imposed,
 the set $A(F)$ 
is $\R$-uniruled. By that he means that for any $a \in A(F)$ there is
a nonconstant polynomial map $g: \R \To \R^m$ (a polynomial curve) such 
that
$g(0)=a$ and $g(t) \in A(F)$ for all $t \in \R$. That in turn
implies that every connected component of $A(F)$ is unbounded
and has positive dimension. 
These results do not hold for everywhere defined rational maps, 
as shown by $y=1/(1+x^2)$, which is proper except at $y=0$.

\section{The birational and {G}alois cases}\label{bi+gal}

\begin{Thm}
Let $F: \R^n \To \R^n$ be a birational nonsingular map. 
Then $F$ has a global inverse, which is also 
a birational nonsingular map. 
\end{Thm}

\begin{proof}
$\R(F)=\R(X)$, so 
the extension degree is $1$. As it bounds the size of all fibers, 
 $F$ is injective, hence invertible.
Thus the rational inverse of $F$ extends to a real analytic map 
on all of $\R^n$. 
Let $g=a/b$ be a component of the inverse, where $a$ 
and $b$ are polynomials with no nonconstant common factor and
suppose $b(x)=0$ for some $x \in \R^n$. 
Let $U$ be an open neighborhood of $x$ in $\C^n$, 
such that $g$ extends to a complex analytic function 
$\tilde{g}$ on $U$ satisfying $b\tilde{g}=A$.
Let $c$ be an irreducible complex polynomial factor
of  $b$ satisfying $c(x)=0$. 
Then $a$ vanishes on the irreducible  hypersurface
$c=0$ in $\C^n$, because it does so in $U$. 
So $c$ is also an irreducible factor of $a$. 
Using complex conjugation, it follows easily that $a^2$ and 
$b^2$ have a nonconstant common factor in the real 
polynomial ring. 
But then so do $a$ and $b$, by unique factorization. 
This contradiction shows that $b$ vanishes nowhere. 
So all components of the inverse are everywhere defined 
rational functions. That makes the inverse an everywhere 
defined rational map, and it is clearly nonsingular and birational. 
\end{proof}

If $F$ is defined over a subfield $k \subset \R$, then so is 
its inverse, since extension degree is preserved by a faithfully 
flat extension of the coefficients. In that case, 
$F$ induces a birational bijection of $k^n$ onto $k^n$. 
Note that $y=x+x^3$ is polynomial, nonsingular, invertible, 
and defined over $\Q$, but the induced map from $\Q$ to
$\Q$ is not surjective.

Remark. In \cite{PolynomialRational}, \p maps $F:\R^n \To \R^n$ 
that map $\R^n$ bijectively onto $\R^n$ are considered, and the question is raised 
of when the inverse is rational. If so, the inverse is everywhere 
defined on $\R^n$ and $F$ is called a polynomial-rational bijection (PRB) of $\R^n$. 
A key technical result is that a \p bijection is a PRB if its 
natural extension to a \p map $\C^n \To \C^n$ maps only 
real points to real points. A PRB $F$ has a nowhere vanishing \J determinant $j(F)$. Conversely, it is shown that a nowhere 
vanishing $j(F)$ alone suffices to establish that a \p map 
$F:\R^n \To \R^n$ of degree two is a bijection and a PRB. 
A related but stronger condition is defined and shown to be sufficient, but not necessary, for \p maps of degree greater than two.

\begin{Thm}
If $F: \R^n \To \R^n$ is a rational nonsingular map and 
$\R(X)/\R(F)$ is a {G}alois extension, then $F$ is 
invertible if, and only if, $F$ is birational. 
\end{Thm}

\begin{proof}

If $F$ is invertible, then the extension has no 
nontrivial automorphisms. So it can be {G}alois only if 
it is of degree $1$. In that (birational) case $F$ does
have an inverse. 
\end{proof}
If $F$ is defined over a subfield $k \subset \R$ and 
$k(X)/k(F)$ is {G}alois, then so is $\R(X)/\R(F)$.

Remark. 
The {G}alois case of the standard JC states that 
a \p Keller map with a {G}alois field extension  
has a polynomial inverse. It was first proved for $k=\C$ only \cite
{GaloisCase}, using methods of the theory of several complex variables. The general characteristic zero case appears
 in \cite{Razar} and, independently, in
\cite{AlgebraicGaloisCase} .
The theorem above is manifestly weaker. 
 Of course, the existence of 
a polynomial inverse implies the triviality of the field extension , so the JC theorem has 
no concrete examples.  

In contrast, in the SRJC and RRJC contexts, the existence of 
an inverse does not imply the field extension  is {G}alois, much less birational. 
For instance, if $y=x+x^3$, the field extension  $\R(y) \subset \R(x)$ 
is neither. 
Even so, a {G}alois extension of degree $d \ne 1$ would represent 
a counterexample to the RRJC of a new, and 
unexpected, type.

\section{Promoted SRJC cases}\label{duo}

The two theorems below have been proved in the SRJC context and,
because of their topological character, they generalize to the RRJC
context almost effortlessly. In both theorems,
let $F: \R^n \To \R^n$ be a rational nonsingular map.
The theorems impose conditions on $A(F)$ that 
are illusory, in that  they conclude
that $F$ is invertible, and so $A(F)$ is actually empty.
For \p $F$, the first theorem was proved by Zbigniew Jelonek
\cite[Theorem 8.2]{geometry}
and the second by
 Christopher I. Byrnes and Anders Lindquist
\cite[Remark 2]{NewProperness}.

\begin{Thm}
If the dimension of $A(F)$ is less than $n-2$,
then $F$ is invertible.
\end{Thm}

\begin{proof}
If $A \subset \R^n$ is a closed \sa set and $\dim A < n-2$,
then $A^c= \R^n \setminus A$ is simply connected
\cite[Lemma 8.1]{geometry}.
This applies to both $A(F)$ and to $B(F)=F^{-1}(A(F))$,
which satisfies $\dim B(F) = \dim  A(F) \cap F(R^n)$. 
The induced map from $B(F)^c$ to $A(F)^c$ is proper,
hence a covering map, and therefore a homeomorphism.
Since  $B(F)$ is not Zariski dense, $F$ is generically injective,
and so invertible.
This proof is that of Jelonek, which simply applies to
rational maps as well.
\end{proof}

\begin{Thm}
If $A(F) \cap F(\R^n) = \emptyset$, then $F$ is invertible.
\end{Thm}

\begin{proof}
The condition states that every point of the (connected, open)
image of $F$ is a point at which $F$ is proper. Equivalently, the
induced map $\R^n \To F(\R^n)$ is proper.
The main result of \cite{NewProperness} is that the standard complex 
JC holds for polynomial maps that are proper as maps onto their image. 
In Remark 2 at the end of the note, that result is also proved 
in the SRJC context. Briefly, $\R^n$ is a 
universal covering space, of finite degree $d$, of $F(\R^n)$. 
By well known results of the branch of topology called P. A. Smith 
theory, there are no fixed point free homeomorphisms of $\R^n$ 
onto itself of prime period. But the fundamental group 
$\pi_1 (F(\R^n))$ is of order $d$, and contains an element of 
prime period unless $d=1$. So $d=1$, $F$ is injective, and 
therefore invertible. The assumption that $F$ 
is polynomial, rather than just real analytic, is used at only two points in 
the proof. First, it ensures that the degree of the covering map is 
finite, and second, that injectivity implies invertibility. 
Rationality is 
sufficient in both situations, so 
this proof works in the RRJC context as well. 
\end{proof}

\end{document}